\documentclass[letterpaper,11pt]{amsart}
\usepackage[utf8]{inputenc}
\usepackage[left=1.5in,right=1.5in,top=1.4in,bottom=1.4in]{geometry}
\usepackage{graphicx} 

\usepackage{amsmath, amssymb, amsthm}
\usepackage{mathtools}
\usepackage{tikz-cd}
\usepackage{soul}
\usepackage[hidelinks]{hyperref}

\newcommand{\cU}{\mathcal{U}}
\newcommand{\cQ}{\mathcal{Q}}
\newcommand{\cR}{\mathcal{R}}

\newcommand{\bC}{\mathbb{C}}
\newcommand{\bR}{\mathbb{R}}
\newcommand{\bF}{\mathbb{F}}
\newcommand{\bN}{\mathbb{N}}
\DeclareMathOperator{\Met}{Met}
\DeclareMathOperator{\Mod}{Mod}

\newcommand{\Cstar}{\ensuremath{\mathrm{C}^*}}
\newcommand{\Wstar}{\ensuremath{\mathrm{W}^*}}

\newcommand\diff{\mathop{}\!\mathrm{d}}

\newcommand{\norm}[1]{{\left\lVert #1\right\rVert}}
\newcommand{\ra}{\ensuremath{\rightarrow}}

\theoremstyle{plain}

\newtheorem{thm}{Theorem}[section]
\newtheorem*{thm*}{Theorem}
\newtheorem{lem}[thm]{Lemma}

\theoremstyle{definition}
\newtheorem{defn}[thm]{Definition}
\newtheorem*{defn*}{Definition}

\newtheorem{remark}[thm]{Remark}

\newtheorem*{question*}{Question}

\usepackage[
backend=biber,
style=alphabetic,
sorting=nyt,
maxbibnames=9
]{biblatex}
\title{Free Independence is not Definable}
\author{William Boulanger, Jakub Curda, Emma Harvey, Yizhi Li, Jennifer Pi}

\begin{document}

\begin{abstract}
    Free independence is an important tool for studying the structure of operator algebras. It is natural to ask from the model-theoretic standpoint whether free independence is captured well in first-order model theory via the notion of a definable set. We prove that pairs of freely independent elements do not form a definable set in the sense of continuous model theory, relative to the theory of both C$^*$-probability spaces and tracial von Neumann algebras.
\end{abstract}

\maketitle

\section{Introduction}
    The notion of a definable set is one of the central objects of study in classical model theory, often used in applications to other areas of mathematics (see for example \cite{PillayModelTheory}). 
    Generalizing the notion of a definable set to the setting of continuous model theory requires careful adjustment for the basic theory to transfer to the continuous setting, and the analysis of definable sets even in basic metric structures is far from complete. Previous works have focused on definable sets and functions in the Urysohn sphere and operators on Hilbert spaces, as well as the crucial notion of spectral gap in tracial von Neumann algebras \cite{goldbringDefinableHilbert, goldbringDefinableUrysohn, goldbring2023spectral}. Additionally, a large number of basic objects of study in the theory of C$^*$-algebras are investigated through the lens of definability in \cite[Chapter~3]{farah2021model}, and most recently minimal and maximal tensor norms were shown to be non-definable \cite{goldbring2025definability-tensornorms}. We continue this broad line of work by investigating the definability of the critical notion of free independence in both C$^*$-probability spaces and tracial von Neumann algebras.

    Free independence is a central definition in the basic theory of free probability, introduced by Voiculescu in the 1980s. This theory has subsequently found numerous applications for understanding the structure of both von Neumann algebras and C$^*$-algebras. 
    Indeed, the best-known results on the question of isomorphism of the free group factors relies heavily on free probability \cite{radulescu1992fundamental, dykema1994interpolated}. It has been also used to demonstrate structural properties of the corresponding C$^*$-algebras \cite{haagerup2005new}.
    Recently, freeness has continued to play a role in C$^*$-algebra theory, contributing to a novel approach for demonstrating that certain algebras have nice regularity properties \cite{robert, amrutam2025strict}.

    In this paper, we demonstrate that freeness is not definable, in the sense that the set of all pairs of elements $(x,y)$ which generate free subalgebras within a C$^*$-probability space (or a tracial von Neumann algebra) is not definable. 
    The result follows from the characterisation of definable sets in terms of ultraproducts, and some important tools from operator algebras involving free complementation in ultrapowers. Crucially, we use Popa's asymptotic freeness \cite{popa} and Connes' characterization of amenable factors \cite{connes1976classification}.

\section{Definable Sets in Continuous Logic}
For completeness, we set out the basics of definable sets in continuous logic, and refer the reader to the exposition in \cite[Section~2]{goldbring2023spectral} and \cite[Section~8]{hart2023introduction} for further information.
For the benefit of the reader not versed in model theory, we remark that understanding the proofs that follow only relies upon the characterization of definable sets given by ultraproducts (see equation (\ref{eq:lift}), where the definable set $X$ is defined as pairs of elements which are freely independent).

Definable sets are broadly those subsets that can be quantified over within a logical framework. For example, if a set $S$ is defined by a formula $\varphi$, then it is easy to express the statement ``for all elements $x$ in $S$,..." simply by writing ``for all $x$, $\varphi(x)$ implies..."\footnote{We are describing the case of classical (discrete) model theory here to give the basic flavour of definability. In the continuous setting, not every set defined by a formula is definable, see \cite[Proposition 3.2.8]{farah2021model}.}. 
Examples of such sets in unital C$^*$-algebras include the sets of self-adjoint elements, unitaries, and projections \cite[Example 3.2.7]{farah2021model}.
Definable sets broaden this notion beyond just those sets exactly defined by formulas.

For the purposes of this paper, we fix $T$ to be the theory of C$^*$-probability spaces or of tracial von Neumann algebras. Let $\Mod(T)$ denote the category of all models of $T$ (i.e. the category of C$^*$-probability spaces or tracial von Neumann algebras respectively). 
Throughout this section we write $d(x,y) = \norm{x - y}$ if we are considering C$^*$-algebras (and $d(x,y) = \norm{x-y}_2$ if we are considering tracial von Neumann algebras).

A $T$-functor is a functor $X\colon \Mod(T) \rightarrow \Met$, where $\Met$ is the category of metric spaces with isometric embeddings as morphisms. Additionally, for some $n \geq 1$, we require that $X(M)$ is a closed subspace of $(M_1)^n$ for each model $M$ of $T$ , where $M_1$ is the norm unit ball\footnote{Formally, $X(M)$ is a closed subspace of an arbitrary tuple of sorts evaluated in $M$, but for simplicity here we restrict always to the unit ball of $M$.} of $M$.  We also require that for any elementary map $\rho: M \rightarrow N$ in $\Mod(T)$, $X(\rho)$ is the restriction of $\rho$ to $X(M)$. 

\begin{defn}
    Suppose $X$ is a $T$-functor with $X(M)$ a closed subspace $(M_1)^n$ of $M$. We say that $X$ is a \emph{$T$-definable set} if 
    for every $\epsilon > 0$, there is a $T$-formula $\phi(\overline{x})$ and $\delta > 0$ such that, for all models $M$ of $T$:
    \begin{itemize}
        \item $X(M) \subseteq Z(\phi^M)$, where $Z(\phi^M)$ denotes the zero-set in $M$ of the formula\footnote{Here we mean usual formulas in continuous metric logic, although we could allow for definable predicates as well; see \cite[Section 3.1]{farah2021model}.} $\phi$, and
        \item For all $a = (a_i)_{i =1}^n \in M^n$, if $\phi^M(\overline{a}) < \delta$,
        \[
            \inf_{\overline{x} \in X(M)} \max_{1 \leq i \leq n} \, d(a_i, x_i) \leq \epsilon 
        \]
    \end{itemize}
\end{defn}

This notion of a definable set ensures that:
\begin{itemize}
    \item $X(M)$ is approximated by formulas $\phi$ in a uniform way which does not rely upon the choice of underlying operator algebra $M$, and
    \item each of the approximating formulas $\phi$ satisfies an ``almost-near" property: when $\phi(\overline{a})$ is small, then $\overline{a}$ can be perturbed slightly to something which actually satisfies the property defined by $X(\cdot)$.
\end{itemize}
In particular, it is not difficult to see that the second item above ensures that definable sets allow lifts from ultraproducts (see also \cite[Theorem~8.2]{hart2023introduction}): 
\begin{equation}
    \label{eq:lift}
    X(N) = \prod_{\cU} X(M_i), \quad \text{where $N = \prod_{\cU}M_i$}\,.
\end{equation}

\section{Non-definability of Freeness}

We wish to consider the notion of free independence, which is defined as follows. Recall that a noncommutative probability space is a unital algebra $A$ equipped with a state $\phi$.
\begin{defn}
    Let $(A, \phi)$ be any noncommutative probability space. A family of unital subalgebras $(A_i)_{i \in I}$ is called \textit{free} if the value of $\phi$ on alternating centred words is always zero. More formally, $\phi(a_1 a_2 \cdots a_n) = 0$ whenever 
    $a_j \in A_{i_j}$ with indices satisfying $i_j \neq i_{j+1}$ for all $j$ and $\phi(a_{i_j}) = 0$ for all $j$.

    We say that two elements $a, b \in A$ are freely independent (or free) if they generate free subalgebras $\mathbb{C}[a]$, $\bC[b]$.
\end{defn}
Free independence in particular ensures that the value of $\phi$ on arbitrary words in the algebras $(A_i)$ is completely determined by the restriction to $A_i$ and the freeness condition. 

In what follows, we consider the $T$-functor
\begin{equation}
\label{eq:defX}
    X(M)=\{(x,y)\in (M_1)^2: x \text{ and }y\text{ are free}\}\,,
\end{equation}
which identifies pairs of free elements in a given algebra $M$. 

In order to prove that $X$ is not definable, we exhibit a C$^*$-probability space (resp. tracial von Neumann algebra) for which (\ref{eq:lift}) does not hold. 
We use the fact that in a tracial von Neumann algebra, freeness of $x$ and $y$ extends to freeness of the von Neumann algebras they generate, W$^*(x)$ and W$^*(y)$ (see for example \cite[Proposition~2.5.6]{voiculescu1992free}). 
Thus, two freely independent elements $x,y $ in a tracial von Neumann algebra $M$ generate the free product subalgebra $\Wstar(x) * \Wstar(y) \subseteq M$.

We provide both proofs here in an elementary fashion; however, it should be noted that the C$^*$-algebraic result follows from the von Neumann algebraic result because the unique trace of a UHF algebra is a definable predicate in its theory, see \cite[Theorem 3.5.5]{farah2021model}.

\subsection{Proof for \Cstar-probability spaces}
\label{sec:Cstar}
Throughout this section and the next, we write $A_\cU$ for the C$^*$-algebraic ultraproduct of a C$^*$-algebra $A$, and $M^\cU$ for the von Neumann-algebraic ultraproduct of a von Neumann algebra $M$.

The proof proceeds by contradiction. Assuming freeness is definable, we produce an embedding of a non-amenable algebra into the hyperfinite II$_1$ factor $\cR$.

Since the universal UHF algebra $\cQ$ is selfless (\cite[Theorem 5.2]{robert}), we may choose $x_\infty\in(\cQ)_1$ self-adjoint with spectrum $\sigma(x_\infty)$ consisting of three points and a self-adjoint element $y\in(\cQ_\cU)_1\setminus\bC$ free from the constant sequence $x=[x_\infty,\dots]\in \cQ_\cU$. Now suppose for a contradiction that freeness was definable, so that $X(\mathcal{Q}_\cU) = X(\mathcal{Q})^\cU$, and so take two representative sequences $(x_n)\subset \cQ$ and $(y_n)\subset\cQ$ such that $x_n$ and $y_n$ are free, and $[x_n]_\cU=x$ and $[y_n]_\cU=y$ inside $\cQ_\cU$. We may further assume that $x_n,y_n$ are self-adjoint for all $n\in\bN$ by replacing $x_n$ by $(x_n+x_n^*)/2$ and doing the same for $y_n$.

\begin{lem}
\label{lemma:fd-subalg}
    For $\cU$-almost all $n$, the $\Cstar$-algebra $\Cstar(x_n)$ generated by $x_n$ contains a unital copy of $\bC^3$ and $\Cstar(y_n)\neq\bC$.
\end{lem}
\begin{proof}
    Note that $\lim_\cU x_n=x_\infty$ in norm and so $\sigma(x_n)$ converges to $\sigma(x_\infty)$ in the Hausdorff distance along the ultrafilter $\cU$. Consequently, for $\cU$-almost all $n$, the spectrum $\sigma(x_n)$ will have at least three connected components and hence $\bC^3$ embeds into $\Cstar(x_n)$. Notice also that for $\cU$-almost all $n$, the spectrum of $y_n$ is not a singleton since $y$ is not a scalar.
\end{proof}

Passing to the GNS-closure of $\cQ$ with respect to its unique trace $\tau$ (which is isomorphic to the hyperfinite II$_1$ factor $\cR$), we have constructed two von Neumann subalgebras $A=\Wstar(x_n)$ and $B=\Wstar(y_n)$ which are free with respect to the unique trace $\tau$ on $\cR$. Moreover, we have arranged that $\bC^3$ embeds into $A$, and $B$ is not the complex numbers. In particular, we can find an isomorphic copy of $\bC^2$ in $B$. 

We have found an embedding of the von Neumann-algebraic free product
\begin{equation}
    (\bC^3,\tau|_A)*(\bC^2,\tau|_B)\subseteq\cR\,.
\end{equation}
The traces on the left hand side are faithful as they are restrictions of the faithful trace $\tau$. Consequently, the left hand side is non-amenable, which is our desired contradiction. We include a sketch of the last well-known assertion for the convenience of the reader.

\begin{lem}
    The free product of the tracial von Neumann algebras $(\bC^3,\tau_1)$ and $(\bC^2,\tau_2)$ is non-amenable whenever the traces $\tau_1, \tau_2$ are faithful.
\end{lem}
\begin{proof}[Sketch of proof]
    Using Dykema's \cite[Theorem 2.3]{dykema}, one can express the free product as a direct sum of an interpolated free group factor $L(\bF_s)$ (see \cite{dykema1994interpolated}) and a finite dimensional commutative von Neumann algebra. A calculation 
    then shows that faithfulness of the traces implies that $s>1$ and so $(A,\tau|_A)*(B,\tau|B)$ is non-amenable since it contains the non-amenable factor $L(\bF_s)$. 
\end{proof}

\subsection{Proof for tracial von Neumann algebras}

The proof in the von Neumann setting is nearly identical - we use Popa's asymptotic freeness result for $\cR$ \cite{popa} instead of selflessness of $\cQ$ and have to work a bit harder to reach the conclusion of Lemma \ref{lemma:fd-subalg}. We yet again choose $x_\infty\in (\cR)_1$ self-adjoint with $\sigma(x_\infty)$ consisting of three points and choose $y\in(\cR^\cU)_1\setminus\bC$ self-adjoint and free from the constant sequence $[x_\infty,\dots]$ using \cite{popa}. Assuming for a contradiction that freeness is definable, so that $X(\mathcal{R}^\cU) = X(\mathcal{R})^\cU$, choose representative sequences $(x_n)\subset \cR$ and $(y_n)\subset\cR$ such that $x_n$ and $y_n$ are free, and $[x_n]_\cU=x$ and $[y_n]_\cU=y$ inside $\cR^\cU$. As before, we may further assume that $x_n,y_n$ are self-adjoint for all $n\in\bN$.

\begin{lem}
    For $\cU$-almost all $n$, the von Neumann algebra $\Wstar(x_n)$ generated by $x_n$ contains a unital copy of $\bC^3$ and $\Wstar(y_n)\neq\bC$.
\end{lem}

\begin{proof}
    We use \cite[Corollary 2]{kittaneh} which states that for two normal Hilbert-Schmidt operators $a,b$ and any Lipschitz function $f\in C(\sigma(a)\cup\sigma(b))$ with Lipschitz constant $k$, we have the inequality
    \begin{equation}
        \norm{f(a)-f(b)}_2\leq k\norm{a-b}_2\,.
    \end{equation}

    We wish to apply the result to the sequence $x_n$ which converges to $x_\infty$ in the 2-norm. Since $x_n$ is bounded we can choose large enough $r\in\bR_+$ such that $\sigma(x_n)\cup \sigma(x_\infty)$ is contained in the disk of radius $r$. 
    Say $\sigma(x_\infty)=\{\lambda_1, \lambda_2, \lambda_3\}$ and fix disjointly supported Lipschitz functions $f_1,f_2,f_3\in C(\bR,\bR_+)$ such that $f_i$ is only supported on a small neighbourhood of $\lambda_i$.
    As a consequence of 2-norm convergence, we then have for each $i$,
    \begin{equation}
        \begin{split}
        \int_{\sigma(x_n)} f_i\diff\mu_{x_n}&=\tau(f_i(x_n)) \\
        &\ra \tau( f_i(x_\infty))
        =\int_{\sigma(x_\infty)} f_i\diff\mu_{x_\infty}>0\,,
    \end{split}
    \end{equation}
    where the measures $\mu_{x_n}$ are the measures on $L^\infty(\sigma(x_n))$ induced by $\tau$.
    In particular we see that for $\cU$-almost all $n$, the support of $f_i$ intersects $\sigma(x_n)$ non-trivially and so the element $x_n$ must have at least three distinct points in its spectrum.

    The fact that $\Wstar(y_n)\neq\bC$ follows again from $y$ not being a scalar.
\end{proof}

The rest of the proof is analogous to the proof given in Section \ref{sec:Cstar}.

\begin{remark}
    The same proof also shows that freeness is not definable, even when relativised to the theory of II$_1$ factors.
\end{remark}

\subsection*{Acknowledgements}
    This work occurred during a 6-week summer undergraduate research project. We thank Trinity College, the Mathematical Institute at Oxford, and Magdalen College for funding WB, EH, and YL respectively. JC and JP were supported by the Engineering and Physical Sciences Research Council (EP/X026647/1).
    
    We gladly thank Austin Shiner for numerous useful discussions and references for this work. We also thank Mira Tartarotti for helpful comments improving the exposition.

\printbibliography

\end{document}